\documentclass[12pt,reqno]{amsart}

\usepackage[T1]{fontenc}
\usepackage[margin=1.1in]{geometry}
\usepackage{amsmath,amssymb,mathtools}
\usepackage{xcolor}
\usepackage{microtype}
\usepackage[
hypertexnames=false,
pdftitle={Rational Spectra for Finite Hadamard Pairs and Bounded Sets on the real Line},
pdfauthor={Xiao-Ye Fu and Zi-Jian Song},
pdfkeywords={Spectral set, rational spectrum, finite exponential Hadamard pair.}
]{hyperref}
\newtheorem{theorem}{Theorem}[section]
\newtheorem{lemma}[theorem]{Lemma}
\newtheorem{proposition}[theorem]{Proposition}

\theoremstyle{definition}
\newtheorem{definition}[theorem]{Definition}

\newtheorem{remark}[theorem]{Remark}

\newcommand{\Q}{\mathbb Q}
\newcommand{\Z}{\mathbb Z}
\newcommand{\Gal}{\operatorname{Gal}}

\newcommand{\rank}{\operatorname{rank}}

\numberwithin{equation}{section}

\title{Rational Spectra for Finite Hadamard Pairs
	and Bounded Spectral Sets on the real line}

\author{Xiao-Ye Fu}
\author{Zi-Jian Song}

\subjclass[2020]{42C15, 42C40, 42A85, 28A80}
\keywords{ Spectral set, rational spectrum, finite exponential Hadamard pair.}

\begin{document}

	\begin{abstract}
		We prove that every finite spectral pair \((A,\Gamma)\) with
		\(A\subset\mathbb Z\), \(\Gamma\subset\mathbb R/\mathbb Z\), and
		\(0\in\Gamma\) has a rational spectrum, that is, \(	\Gamma\subset\mathbb Q/\mathbb Z.\)
	Our argument relies on a modulus rigidity theorem for generalized
		Vandermonde systems satisfying inverse-orthogonality relations, constructed by hyperbolic positive-definite kernels. Combined
		with Galois conjugation and Kronecker's theorem, this rigidity forces the
		associated exponential nodes to be roots of unity. As an application,
		using the periodicity and fiberization of one-dimensional spectra, we
		show that every spectrum \(\Lambda\) of a bounded measurable spectral set
		\(\Omega\subset\mathbb R\) with \(|\Omega|=1\) and \(0\in\Lambda\) is
		contained in \(\mathbb Q\). This rational-spectrum result completes a chain of equivalences proven by Dutkay and Lai, which reduces the full one-dimensional Fuglede's conjecture to its finite cyclic analogues over \(\mathbb{Z}_n\).
		 %
	\end{abstract}
	\maketitle

\section{Introduction}
\label{sec:introduction}

Let \(\Omega\subset\mathbb R^d\) be a bounded measurable set with positive Lebesgue measure. We call \(\Omega\) a \emph{spectral set} if there exists a countable set \(\Lambda\subset\mathbb R^d\) such that
\[
E(\Lambda):=\left\{e_\lambda(x):=e^{2\pi i\langle\lambda,x\rangle}:\lambda\in\Lambda\right\}
\]
forms an orthogonal basis of \(L^2(\Omega)\). In this case, \(\Lambda\) is called a \emph{spectrum} of \(\Omega\), and \((\Omega,\Lambda)\) is called a \emph{spectral pair}.

\vspace{0.2cm}

The notion of spectral sets was introduced by Fuglede \cite{Fuglede1974} in his study of commuting self-adjoint extensions of partial differential operators. He conjectured that a bounded measurable set in \(\mathbb R^d\) is spectral if and only if it tiles \(\mathbb R^d\) by translations. While this equivalence holds in the lattice setting considered by Fuglede's original framework, counterexamples have since been constructed for general measurable sets in dimensions \(d\geq 3\) \cite{Tao2004,Matolcsi2005,KolountzakisMatolcsi2006,FarkasMatolcsiMora2006}. The conjecture remains open in dimensions one and two (see \cite{Kolountzakis2024} for a recent overview).

\vspace{0.2cm}

The one-dimensional problem has a particularly rigid structure. Among the early positive results, Laba proved Fuglede's conjecture
for unions of two intervals \cite{Laba2001}. For finite unions of
intervals, Bose and Madan established the periodicity of every spectrum
\cite{BoseMadan2011}, and Kolountzakis subsequently gave a shorter
proof \cite{Kolountzakis2012}. Iosevich and Kolountzakis later extended
this periodicity theorem to arbitrary bounded measurable spectral sets
on the real line \cite{IosevichKolountzakis2013}. More recently, weak
tiling has provided a complementary geometric restriction on spectral
sets \cite{KolountzakisLevMatolcsi2023,LevMatolcsi2022}.

\vspace{0.2cm}

Consequently, a direct application of these periodicity results is that every spectrum of a bounded spectral set on the line is determined by finitely many representatives together with a periodic repetition. Periodicity alone, however, does not determine the arithmetic nature of those representatives, and it is natural to ask whether spectrality forces them to be rational.

%

\vspace{0.2cm}

The importance of this question was made precise by Dutkay and Lai \cite{DutkayLai2014}, who compared the spectral-set conjectures on \(\mathbb R\), \(\mathbb Z\), and the finite cyclic groups \(\mathbb Z_n\). Write \(\mathrm{S\!-\!T}(G)\) for the assertion that every bounded spectral set in \(G\) tiles \(G\), and \(\mathrm{T\!-\!S}(G)\) for the converse assertion. Here, a statement involving \(\mathbb Z_n\) is understood to hold for every \(n\in\mathbb N\). They proved
\[
\mathrm{T\!-\!S}(\mathbb R)
\Longleftrightarrow
\mathrm{T\!-\!S}(\mathbb Z)
\Longleftrightarrow
\mathrm{T\!-\!S}(\mathbb Z_n),
\]
where the finite cyclic formulation is also equivalent to the corresponding universal spectrum conjecture \cite{FarkasMatolcsiMora2006, LagariasWang1997}. For the opposite direction, they established
\[
\mathrm{S\!-\!T}(\mathbb R)
\Longrightarrow
\mathrm{S\!-\!T}(\mathbb Z)
\Longrightarrow
\mathrm{S\!-\!T}(\mathbb Z_n),
\]
and proved that the reverse implications follow provided every bounded spectral set of Lebesgue measure one admits a rational spectrum. 
They also showed that a full positive resolution of Fuglede conjecture on \(\mathbb R\) would itself imply the existence of such rational spectra.

\vspace{0.2cm}

Rationality provides the critical arithmetic bridge that reduces the continuous spectral problem to finite cyclic group analysis. After using periodicity and the fiberization decomposition developed in Section \ref{sec:fiberization}, the continuous problem restricts to a finite frequency set \(\Gamma\subset\mathbb R/\mathbb Z\). If \(\Gamma\subset\mathbb Q/\mathbb Z\), then a common denominator realizes these frequencies as characters of some finite cyclic group \(\mathbb Z_N\). The resulting exponential orthogonality relations become tractable through finite group representation theory. Thus rationality is not merely a refinement of periodicity; it is what allows the one-dimensional continuous problem to be reduced to its finite cyclic analogue. 

\vspace{0.2cm}

This connection is particularly relevant in view of the progress on Fuglede's conjecture for cyclic groups. Laba's work on the multiplicative structure of roots of mask polynomials provided an
early link between integer tilings, cyclotomic divisibility, and spectrality \cite{Laba2002}. The full conjecture has been proved for cyclic groups of order \(p^nq\), \(pqr\), and \(pqrs\), where the
letters denote distinct primes \cite{LabaLondner2023, MalikiosisKolountzakis2017,Shi2019,
 KissMalikiosisSomlaiVizer2022}. Malikiosis also obtained further structural results and additional cases for groups of order \(p^mq^n\) \cite{Malikiosis2022}. These works show that roots of unity
and cyclotomic structure are central features of the finite cyclic
problem.

\vspace{0.2cm}

The finite-dimensional problem arising in this reduction is also of independent interest. Let \(A\subset\mathbb Z\) and \(\Gamma\subset\mathbb R/\mathbb Z\) be finite sets with \(\# A =\# \Gamma=d.\)
Then \((A,\Gamma)\) is a finite spectral pair precisely when
\[
\frac{1}{\sqrt d}
\left(e^{2\pi i a\gamma}\right)_{a\in A,\ \gamma\in\Gamma}
\]
is unitary. Hence every finite spectral pair gives rise to a complex Hadamard matrix. Such matrices and finite spectral pairs have played an important role in the study of Fuglede's conjecture; see \cite{LagariasWang1997,KolountzakisMatolcsi2006}. Finite Hadamard
structures also occur naturally in the construction of self-affine
spectral measures \cite{DutkayHaussermannLai2019}.  General complex Hadamard matrices may occur in continuous parameter families and need not have roots of unity as their entries \cite{TadejZyczkowski2006}. The matrices considered here, however, possess an additional generalized Vandermonde structure
\[
\left(z_j^a\right)_{a\in A,\ 0\leq j<d},
\qquad |z_j|=1,
\]
with integer exponents, which enforces significantly stronger rigidity constraints. 

\vspace{0.2cm}

 The proof of our main rationality theorem will rely on a fundamental concircularity rigidity theorem for inverse-orthogonal systems, which we isolate here for clarity. It asserts that any family of nonzero complex numbers satisfying the inverse-orthogonality relations must have a common modulus. 
 
\vspace{0.2cm}

\begin{theorem}
	\label{thm:concyclic-rigidity}
	Let \(A\subset\mathbb Z\) be a finite set with \(\# A=d\), and let
	\(w_1,\ldots,w_d\in\mathbb C^\times\), where \(\mathbb C^\times = \mathbb C\setminus \{0\}\),  satisfy
	\begin{equation}
		\sum_{a\in A}
		\left(\frac{w_j}{w_i}\right)^a
		=
		d\delta_{ij},
		\qquad 1\le i,j\le d.
		\label{eq:inverse-orthogonality}
	\end{equation}
	Then
	\[
	|w_1|=\cdots=|w_d|.
	\]
\end{theorem}

\vspace{0.2cm}

Theorem \ref{thm:concyclic-rigidity} is the central analytical tool of this paper, which is proved by employing a  hyperbolic positive-definite kernel and the Schur product theorem. It yields the uniform modulus property that, when combined with Galois automorphisms and Kronecker's theorem, supports  our rationality arguments. This modulus-uniformity rigidity provides an essential algebraico-analytic link between the discrete orthogonality conditions for finite spectral pairs and the rationality of spectral frequencies appearing in our main theorems. This modulus-uniformity rigidity theorem supplies an essential algebraic-analytic bridge connecting the discrete orthogonality condition for finite spectral sets to the rationality of spectral elements in our main results. We can now state our first main result.

\vspace{0.2cm}

For a finite set \(A\subset\mathbb Z\), define
\[
F_A(t):=\sum_{a\in A}e^{2\pi iat},
\qquad t\in\mathbb R/\mathbb Z.
\]

\begin{theorem}
	\label{thm:finite-rationality}
	Let \(A\subset\mathbb Z\) be a finite spectral set with \(\# A =d\), and let
	\[
	\Gamma=\{\gamma_0,\ldots,\gamma_{d-1}\}
	\subset\mathbb R/\mathbb Z,
	\qquad \gamma_0=0 
	\]
be its spectrum, that is,
	\[
	F_A(\gamma_j-\gamma_i)=0,
	\qquad i\neq j.
	\]
	Then
	\(
	\Gamma\subset\mathbb Q/\mathbb Z.
	\)
	More precisely, there exists \(N\in\mathbb N\) such that
	\(
	\Gamma\subset\frac1N\mathbb Z/\mathbb Z.
	\)
\end{theorem}

\vspace{0.2cm}

We begin the proof by demonstrating that each number \(z_j=e^{2\pi i\gamma_j}\) is an algebraic integer. After applying a suitable Galois automorphism, the conjugate nodes satisfy the inverse-orthogonality relations as in Theorem \ref{thm:concyclic-rigidity}. The modulus rigidity Theorem \ref{thm:concyclic-rigidity}  then forces these conjugate nodes to have equal modulus. Since \(z_0=1\), every algebraic conjugates of \(z_j\) lies on the unit circle, and Kronecker's theorem then implies that each \(z_j\) is a root of unity.

\vspace{0.2cm}

Our second main result extends this finite-dimensional rigidity framework to bounded measurable spectral sets on the real line. Bose and Madan \cite{BoseMadan2018} obtainded partial positive progress on the one-dimensional rational spectrum conjecture: they verified spectral rationality for special periodic spectral sets consisting of finite unions of unit intervals and introduced flag-type criteria to test rationality for periodic spectra on \(\mathbb{R}\). We point out that Theorem \ref{thm:bounded-rationality} below was previously claimed in \cite{Zhou2024}. Nevertheless, 
 the critical summation formula on page \(9\) of \cite{Zhou2024} incorrectly supposes an entire arithmetic progression is contained inside the spectral set \(\Omega\), a claim disproven by an explicit counterexample provided by the reviewer. This formula is indispensable to the rest of the proof and the resulting gap cannot be fixed.
 In the present paper, we resolve this deficiency by employing Galois automorphisms and modular rigidity arguments.

\begin{theorem}
	\label{thm:bounded-rationality}
	Let \(\Omega\subset\mathbb R\) be a bounded measurable spectral set with
	\(
	|\Omega|=1,
	\)
	and let \(\Lambda\) be a spectrum of \(\Omega\) with \(0\in\Lambda\). Then
	\(
	\Lambda\subset\mathbb Q.
	\)
\end{theorem}

The proof combines the periodicity property of one-dimensional spectra with the fiberization decomposition constructed in Section \ref{sec:fiberization}.
By periodicity, we write \(\Lambda=S+T\mathbb Z\) with \(S\subset[0,T)\), the fiberization theorem imples that \(S/T\) is a spectrum for a finite subset of \(\mathbb Z\). Applying Theorem~\ref{thm:finite-rationality} gives \(S/T\subset\mathbb Q/\mathbb Z\), so \(S\subset\mathbb Q\).  It follows that \(\Lambda\subset\mathbb Q\) since \(T\in\mathbb N\).

\vspace{0.2cm}

Theorem~\ref{thm:bounded-rationality} supplies the rational-spectrum
hypothesis in the reduction theorem of Dutkay and Lai \cite{DutkayLai2014}, and in fact
shows that every normalized spectrum is rational. Consequently, by
\cite[Theorem~1.3(ii)]{DutkayLai2014},
\[
\mathrm{S\!-\!T}(\mathbb R)
\Longleftrightarrow
\mathrm{S\!-\!T}(\mathbb Z)
\Longleftrightarrow
\bigl[
\mathrm{S\!-\!T}(\mathbb Z_n)
\text{ holds for every }n\in\mathbb N
\bigr].
\]
Together with their corresponding equivalence for the
tile-to-spectral direction, this reduces the one-dimensional Fuglede
conjecture to its finite cyclic analogues.

\vspace{0.2cm}

The paper is organized as follows. In Section~2, we collect all required preliminaries: finite spectral pairs, complex Hadamard matrices, periodicity of spectra, positive semidefinite matrices and positive-definite functions, and some background on algebraic number theory including Galois automorphisms and Kronecker’s theorem. In Section~3, we establish the fiberization theorem for periodic spectra, which reduces continuous spectral problems into almost-everywhere finite-dimensional fiberwise subproblems.  In Section~4, we construct the hyperbolic positive-definite kernel and prove the core modulus rigidity theorem \ref{thm:concyclic-rigidity}. In Section~5, we complete the proofs of Theorem~\ref{thm:finite-rationality} and Theorem~\ref{thm:bounded-rationality}.

\section{Preliminaries}
\label{sec:preliminaries}
This section collects all standard notation, fundamental definitions, and auxiliary results required for our analysis.
Throughout the paper, we write
\[
e_\lambda(x)=e^{2\pi i\lambda x},
\qquad
\mathbb T=\{z\in\mathbb C:|z|=1\},
\]
and denote by \(\delta_{ij}\) the  delta function.

\subsection{Finite spectral pairs and periodicity}

We begin with the finite-dimensional notion of spectrality that will 
arise from fiberizing bounded spectral sets on \(\mathbb {R}\).

\begin{definition}[Finite spectral pair]
	\label{def:finite-spectral-pair}
	Let \(A\subset\mathbb Z\) and
	\(\Gamma\subset\mathbb R/\mathbb Z\) be finite sets with \(\# A = \#\Gamma=d\).
	We say that \(\Gamma\) is a \emph{spectrum} of \(A\), or that
	\((A,\Gamma)\) is a \emph{finite spectral pair}, if the vectors
	\[
	\frac{1}{\sqrt d}
	\bigl(e^{2\pi i a\gamma}\bigr)_{a\in A},
	\quad \gamma\in\Gamma,
	\]
	form an orthonormal basis of \(\ell^2(A)\). Equivalently,
	\[
	\frac{1}{\sqrt d}
	\left(e^{2\pi i a\gamma}\right)_{a\in A,\ \gamma\in\Gamma}
	\]
	is a unitary matrix.
\end{definition}

Thus \((A,\Gamma)\) is a finite spectral pair if and only if
\[
\sum_{a\in A}e^{2\pi ia(\gamma-\gamma')}
=
d\,\delta_{\gamma,\gamma'},
\quad \gamma,\gamma'\in\Gamma.
\]

\begin{definition}[Complex Hadamard matrix]
	A matrix \(H\in M_d(\mathbb C)\) is called a \emph{complex Hadamard
		matrix} if every entry of \(H\) has modulus \(1\) and
	\[
	H^*H=dI_d.
	\]
	Equivalently, \(d^{-1/2}H\) is unitary.
\end{definition}

The matrix in Definition \ref{def:finite-spectral-pair} is a concrete example of a complex Hadamard matrix. The periodicity theorem below furnishes the fundamental structural property for bounded spectral subsets of the real line and underlies our fiberization construction in Section 3.

\begin{theorem}[
	\cite{IosevichKolountzakis2013}]
	\label{thm:one-dimensional-spectrum-periodic}
	Let \(\Omega\subset\mathbb R\) be a bounded measurable set with
	\(0<|\Omega|<\infty\). If \((\Omega,\Lambda)\) is a spectral pair, then
	\(\Lambda\) is periodic, that is,
	\[
	\Lambda=T\mathbb Z+\{\lambda_1,\ldots,\lambda_m\}
	\]
	for some \(T>0\) and finitely many \(\lambda_1,\ldots,\lambda_m\in\mathbb R\).
\end{theorem}

\subsection{Positive semidefinite matrices and positive-definite functions}

For a complex matrix \(K\), we denote its conjugate
transpose by \(K^*\).

\begin{definition}
	A matrix \(K\in M_d(\mathbb C)\) is \emph{Hermitian} if
	$
	K^*=K.
	$
	A Hermitian matrix \(K\) is \emph{positive semidefinite}, written
	\(K\succeq0\), if
	\[
	c^*Kc\geq0,
	\quad\text{for every }c\in\mathbb C^n.
	\]
\end{definition}

If \(K=(K_{pq})_{p,q=1}^d\) and
\(c=(c_1,\ldots,c_d)^T\), then
\[
c^*Kc
=
\sum_{p,q=1}^d
\overline{c_p}K_{pq}c_q.
\]
By the spectral theorem for Hermitian matrices, a Hermitian matrix is positive semidefinite
if and only if all of its eigenvalues are nonnegative.

\begin{definition}[Hadamard product]
	Let \(K=(K_{pq})\) and \(L=(L_{pq})\) be matrices of the same size.
	Their \emph{Hadamard product}, also called the \emph{Schur product}, is
	the entrywise product
	\[
	K\circ L
	=
	(K_{pq}L_{pq})_{p,q}.
	\]
	This operation is different from ordinary matrix multiplication.
\end{definition}

\begin{theorem}[Schur product theorem \cite{Horn:2013:MatrixAnalysis}]
	\label{thm:schur-product}
	If \(K,L\in M_d(\mathbb C)\) are Hermitian positive semidefinite
	matrices, then
	\[
	K\circ L\succeq0.
	\]
\end{theorem}

\begin{definition}[Positive-definite function]
	A function \(f:\mathbb R\to\mathbb C\) is called
	\emph{positive definite} if, for every \(d\geq1\), every
	\(x_1,\ldots,x_d\in\mathbb R\), and every
	\(c_1,\ldots,c_d\in\mathbb C\),
	\[
	\sum_{p,q=1}^d
	\overline{c_p}c_q\,f(x_p-x_q)\geq0.
	\]
	Equivalently, for every finite sequence
	\(x_1,\ldots,x_d\), the matrix
	\[
	\bigl(f(x_p-x_q)\bigr)_{p,q=1}^d
	\]
	is Hermitian positive semidefinite.
\end{definition}

The following Fourier--Stieltjes criterion gives a standard source of
positive-definite functions.

\begin{proposition}
	\label{prop:fourier-stieltjes-positive-definite}
	Let \(\mu\) be a positive finite Borel measure on \(\mathbb R\), and
	define
	\[
	f(x)=\int_{\mathbb R}e^{i\xi x}\,d\mu(\xi).
	\]
	Then \(f\) is a continuous positive-definite function on
	\(\mathbb R\).
\end{proposition}

\subsection{Algebraic preliminaries} We recall basic concepts from algebraic number theory and Galois theory, which are essential to our rigidity and rationality proofs.

\begin{definition}
	A complex number \(\alpha\) is called an \emph{algebraic integer} if
	it is a root of a monic polynomial in \(\mathbb Z[X]\). Its
	\emph{minimal polynomial} is the unique monic irreducible polynomial
	in \(\mathbb Q[X]\) having \(\alpha\) as a root. The other roots of
	this polynomial are called the \emph{algebraic conjugates} of
	\(\alpha\).
\end{definition}

\begin{definition}
	A complex number \(\zeta\) is called a \emph{root of unity} if \(\zeta^N=1\)
	for some \(N\in\mathbb N\).
\end{definition}

\vspace{0.2cm}

If \(K/\mathbb Q\) is a finite normal extension, then \(\operatorname{Gal}(K/\mathbb Q)\)
denotes the group of field automorphisms of \(K\) that fix
\(\mathbb Q\) pointwise. Such automorphisms preserve sums, products,
inverses, and integer powers. If \(K\) is normal and contains an
algebraic number \(\alpha\), then every algebraic conjugate of
\(\alpha\) is of the form
\[
\sigma(\alpha)
\quad
\text{for some }
\sigma\in\operatorname{Gal}(K/\mathbb Q).
\]

\vspace{0.2cm}

We end this subsection with Kronecker’s classical theorem, the key algebraic result used to deduce rationality in the proof of
Theorem~\ref{thm:finite-rationality}.

\begin{theorem}[Kronecker's Theorem]
	\label{lem:kronecker}
	Let \(\alpha\) be an algebraic integer. If every algebraic conjugate
	of \(\alpha\) has modulus at most \(1\), then either \(\alpha=0\) or
	\(\alpha\) is a root of unity. In particular, if every algebraic
	conjugate of \(\alpha\) has modulus \(1\), then \(\alpha\) is a root
	of unity.
\end{theorem}

	\section{Fiberization of a Periodic Spectrum}
	\label{sec:fiberization}
	
	Let \(\Omega\subset\mathbb R\) be a bounded measurable spectral set
	with \(|\Omega|=1\), and let \(\Lambda\) be a spectrum of \(\Omega\)
	containing \(0\). By Theorem~\ref{thm:one-dimensional-spectrum-periodic},
	there exist \(T>0\) and a finite set \(S\subset[0,p)\) with
	\(0\in S\), such that
	\[
	\Lambda=S+T\mathbb Z.
	\]
	By the density theorem for spectra \cite{Landau1967},
	\(D(\Lambda)=|\Omega|=1\), whereas
	\[
	D(\Lambda)=\frac{\#S}{T}.
	\]
	Hence \(T=\#S\) is a positive integer. 
		We therefore write
		\begin{equation}
			\Lambda=S+T\mathbb Z,
			\qquad
			S\subset[0,T),
			\qquad
			0\in S,
			\qquad
			\#S=T.
			\label{eq:periodic-decomposition}
	\end{equation}
		Set \(I=[0,1/T)\). For \(x\in I\), define the integer fiber
	\[
	A_x
	=
	\left\{
	n\in\mathbb Z:
	x+\frac{n}{T}\in\Omega
	\right\}.
	\]
	Since \(\Omega\) is bounded, the sets \(A_x\) are finite and their
	cardinalities are uniformly bounded.

\vspace{0.2cm}

	Consider the Hilbert space
	\[
	\mathcal H
	=
	\left\{
	g:I\times\mathbb Z\to\mathbb C:
	g(x,n)=0\ \text{if }n\notin A_x,\quad
	\int_I\sum_{n\in\mathbb Z}|g(x,n)|^2\,dx<\infty
	\right\},
	\]
	with the natural norm, and define
	\[
	U:L^2(\Omega)\longrightarrow\mathcal H,
	\qquad
	(Uf)(x,n)
	=
	f\left(x+\frac{n}{T}\right),
	\quad n\in A_x.
	\]
	Extending \(f\) by zero outside \(\Omega\), the partition
	\[
	\mathbb R
	=
	\bigsqcup_{n\in\mathbb Z}
	\left[\frac{n}{T},\frac{n+1}{T}\right)
	\]
	gives
	\[
	\|Uf\|_{\mathcal H}^2
	=
	\int_I\sum_{n\in A_x}
	\left|f\left(x+\frac{n}{T}\right)\right|^2dx
	=
	\|f\|_{L^2(\Omega)}^2.
	\]
	Every point of \(\Omega\) has a unique representation
	\(x+n/T\), with \(x\in I\) and \(n\in A_x\), so \(U\) is unitary.
	For \(s\in S\), let
	\[
	v_s(x)
	=
	\left(e^{2\pi isn/T}\right)_{n\in A_x}
	\in\ell^2(A_x).
	\]
	For \(s\in S\) and \(k\in\mathbb Z\), 	since \(e^{2\pi ikn}=1\),
	\begin{equation}
		(Ue_{s+Tk})(x,n)
		=
		e^{2\pi i(s+Tk)x}e^{2\pi isn/T}.
		\label{eq:fiberized-exponential}
	\end{equation}

\vspace{0.2cm}

To analyze the rationality of spectra for bounded spectral sets on the real line, we first reduce the continuous spectral problem to a family of finite-dimensional spectral subproblems by a fiberwise decomposition technique. The upcoming fiberization theorem establishes that for any normalized bounded spectral set \(\Omega\), each fiber \(A_x\) derived from the scaling partition forms a finite spectral set almost everywhere. The proof of this fiberization theorem can be found in \cite{DutkayJorgensen}. We restate and supply its full proof here for self-containment of the manuscript. 

\vspace{0.2cm}
	
	\begin{theorem}[Fiberization theorem]
		\label{thm:fiberization}
			Let \(\Omega\subset\mathbb R\) be a bounded measurable spectral set
		with \(|\Omega|=1\). For the notations defined above and for
		 almost every \(x\in I\),
		\[
		\#A_x=T,
		\]
		and \(A_x\) is a finite spectral set with a spectrum
		\[
		\frac{S}{T}
		=
		\left\{\frac{s}{T}:s\in S\right\}
		\subset\mathbb R/\mathbb Z.
		\]
		Equivalently,
		\(
		\frac1{\sqrt T}
		\left(e^{2\pi isn/T}\right)_{n\in A_x,\ s\in S}
		\)
		is unitary.
	\end{theorem}
	
	\begin{proof}
		Fix \(s\in S\). For two distinct integer \(k\neq k'\), the orthogonality of the exponentials
		\(e_{s+Tk}\) and \(e_{s+Tk'}\) in \(L^2(\Omega)\), together with
		\eqref{eq:fiberized-exponential}, gives
		\[
		\int_I
		\#A_x\,e^{2\pi iT(k-k')x}\,dx
		=
		0.
		\]
		Thus all nonzero Fourier coefficients of the bounded function
		\(x\mapsto\#A_x\) on \(I\) vanish, so \(\#A_x\) is constant almost
		everywhere. Moreover,
		\[
		\int_I\#A_x\,dx
		=
		\sum_{n\in\mathbb Z}
		\int_I
		\mathbf 1_\Omega\left(x+\frac nT\right)dx
		=
		|\Omega|
		=
		1.
		\]
		Since \(|I|=1/T\), it follows that
		\begin{equation}
			\#A_x=T
			\qquad\text{for almost every }x\in I.
			\label{eq:fiber-cardinality}
		\end{equation}
		
\

		Now let \(s,s'\in S\) with \(s\neq s'\). For arbitrary
		\(k,k'\in\mathbb Z\), orthogonality of 	\(e_{s+Tk}\) and \(e_{s+Tk'}\) gives
			\begin{equation}
		\int_I
		e^{2\pi i[(s-s')+T(k-k')]x}
		\sum_{n\in A_x}
		e^{2\pi i(s-s')n/T}\,dx
		=
		0.
		\label{eq:fiber-Fourier-coefficient}
		\end{equation}
	Define	
\[
		G_{s,s'}(x)
		=
		e^{2\pi i(s-s')x}
		\sum_{n\in A_x}
		e^{2\pi i(s-s')n/T}		
	\]
	Then	\(G_{s,s'}\in L^1(I)\) by \eqref{eq:fiber-cardinality}, and the preceding identity \eqref{eq:fiber-Fourier-coefficient} asserts that every Fourier coefficient of	\(G_{s,s'}(x)\) vanishes.  The uniqueness theorem for Fourier series
		therefore yields \(G_{s,s'}=0\) almost everywhere on \(I\). Since the first
		factor \(	e^{2\pi i(s-s')x}\) never vanishes,
			\begin{equation}
		\sum_{n\in A_x}e^{2\pi i(s-s')n/T}=0
			\label{eq:orthogonality}
			\end{equation}
		for almost every \(x\in I\).
		
		\
		
	Since \(S\) is finite, the relations \eqref{eq:orthogonality} and
		\eqref{eq:fiber-cardinality} hold simultaneously on a common
		full-measure subset of \(I\). Therefore
		\[
		\sum_{n\in A_x}
		e^{2\pi i(s-s')n/T}
		=
		T\delta_{s,s'},
		\qquad s,s'\in S.
		\]
		Thus the \(T\) vectors
		\[
		\frac1{\sqrt T}
		\left(e^{2\pi isn/T}\right)_{n\in A_x},
		\qquad s\in S,
		\]
		form an orthonormal family in the \(T\)-dimensional space
		\(\ell^2(A_x)\), and hence an orthonormal basis.
	\end{proof}

\section{Hyperbolic positive-definite kernels and modulus rigidity}
\label{sec:finite-rationality}

In this section, we prove Theorem~\ref{thm:concyclic-rigidity}. For a
finite spectral set \(A\subset\mathbb Z\) with \(\#A=d\), Let
\[
\Gamma=\{\gamma_0,\ldots,\gamma_{d-1}\}
\subset\mathbb R/\mathbb Z,
\quad
\gamma_0=0,
\] 
be a spectrum of \(A\). Define
\[
F_A(t)=\sum_{a\in A}e^{2\pi iat},
\quad t\in\mathbb R/\mathbb Z.
\]
Then
\[
F_A(\gamma_j-\gamma_i)=0
\qquad (i\neq j).
\]
Since \(F_A(0)=d\), this is equivalent to
\begin{equation}
	F_A(\gamma_j-\gamma_i)
	=
	d\delta_{ij},
	\quad 0\leq i,j< d.
	\label{eq:finite-orthogonality}
\end{equation}

Set
\[
z_j=e^{2\pi i\gamma_j}\in\mathbb T,
\qquad
\mathbb T=\{z\in\mathbb C:|z|=1\}.
\]
Then \eqref{eq:finite-orthogonality} becomes
\begin{equation}
	\sum_{a\in A}
	\left(\frac{z_j}{z_i}\right)^a
	=
	d\delta_{ij},
	\quad 0\leq i,j< d.
	\label{eq:finite-inverse-orthogonality}
\end{equation}
Equivalently, for
\(
H=(z_j^a)_{a\in A,\ 0\leq j<d},
\)
since \(\overline{z_i^a}=z_i^{-a}\), one has
\[
H^*H=dI_d.
\]
Thus \(d^{-1/2}H\) is unitary.

\vspace{0.2cm}

The algebraic relations
\eqref{eq:finite-inverse-orthogonality} are preserved under Galois automorphisms, although the conjugate nodes need not remain on the unit circle. This is precisely the setting of Theorem \ref{thm:concyclic-rigidity}, which asserts that any nonzero complex nodes satisfying these relations must have a common modulus.  Its proof relies on the following hyperbolic positive-definite kernel.

\begin{lemma}
	\label{lem:hyperbolic-sine-kernel}
	Let \(0<t<T\), and define
	\[
	k_{t,T}(x)
	=
	\begin{cases}
		\dfrac{\sinh(tx)}{\sinh(Tx)},&x\neq0,\\[6pt]
		\dfrac{t}{T},&x=0.
	\end{cases}
	\]
	Then \(k_{t,T}\) is a continuous positive-definite function on
	\(\mathbb R\). Equivalently, for every finite sequence
	\(x_1,\ldots,x_n\in\mathbb R\), the matrix
	\[
	\bigl(k_{t,T}(x_p-x_q)\bigr)_{p,q=1}^n
	\]
	is Hermitian positive semidefinite.
\end{lemma}

\begin{proof}
	The value at the origin is the continuous extension, since
	\[
	\lim_{x\to0}\frac{\sinh(tx)}{\sinh(Tx)}=\frac{t}{T}.
	\]
	Moreover, \(k_{t,T}\) is real-valued and even.
		In fact, we have  the Fourier integral formula
	\begin{equation}
		\int_{\mathbb R}
		\frac{e^{isu}}{\cosh u+\cos a}\,du
		=
		\frac{2\pi\sinh(as)}
		{\sin a\,\sinh(\pi s)},
		\qquad 0<a<\pi,~s\in\mathbb{R},
		\label{eq:hyperbolic-fourier-first}
	\end{equation}
	where the value at \(s=0\) is understood by continuity. For a detailed calculation of the integral \eqref{eq:hyperbolic-fourier-first}, see Appendix A.	
	In \eqref{eq:hyperbolic-fourier-first}, set
	\[
	a=\frac{\pi t}{T},
	\qquad
	u=\frac{\pi\xi}{T},
	\qquad
	s=\frac{Tx}{\pi}.
	\]
	Since \(du=(\pi/T)\,d\xi\), we obtain
	\[
	k_{t,T}(x)
	=
	\frac{\sin(\pi t/T)}{2T}
	\int_{\mathbb R}
	\frac{e^{i\xi x}}
	{\cosh(\pi\xi/T)+\cos(\pi t/T)}
	\,d\xi.
	\]
	Define
	\[
	d\mu_{t,T}(\xi)
	=
	\frac{\sin(\pi t/T)}{2T}
	\frac{d\xi}
	{\cosh(\pi\xi/T)+\cos(\pi t/T)}.
	\]
	Since \(0<t<T\), the density of \(\mu_{t,T}\) is positive, and its exponential decay
	at infinity shows that \(\mu_{t,T}\) is a positive finite Borel
	measure. Thus
	\[
	k_{t,T}(x)
	=
	\int_{\mathbb R}e^{i\xi x}\,d\mu_{t,T}(\xi)
	\]
	is a continuous positive-definite function on \(\mathbb {R}\) by Proposition \ref{prop:fourier-stieltjes-positive-definite}.
	Consequently, for arbitrary \(c_1,\ldots,c_n\in\mathbb C\),
	\[
	\begin{aligned}
		\sum_{p,q=1}^n
		\overline{c_p}c_q\,k_{t,T}(x_p-x_q)
		&=
		\int_{\mathbb R}
		\left|
		\sum_{q=1}^n c_qe^{-i\xi x_q}
		\right|^2
		\,d\mu_{t,T}(\xi)\geq0.
	\end{aligned}
	\]
	Therefore the matrix \((k_{t,T}(x_p-x_q))_{p,q=1}^n\) is Hermitian positive semidefinite.
\end{proof}

\vspace{0.2cm}

The hyperbolic positive-definite kernel constructed in Lemma \ref{lem:hyperbolic-sine-kernel} will serve as a key analytical obstruction tool for our subsequent rigidity argument. Using this kernel family, We now proceed to prove Theorem \ref{thm:concyclic-rigidity}.

\vspace{0.2cm}
 
\begin{proof}[\bf Proof of Theorem~\ref{thm:concyclic-rigidity}]
	Define the \(d\times d\) matrices
	\[
	M=(w_j^a)_{a\in A,\,1\le j\le d},
	\qquad
	M^\vee=(w_j^{-a})_{a\in A,\,1\le j\le d}.
	\]
	Condition \eqref{eq:inverse-orthogonality} is equivalent to
	\[
	(M^\vee)^TM=dI_d.
	\]
	Since \(M\) is a square matrix, \(d^{-1}(M^\vee)^T\) is also its right inverse.
	Thus
	\[
	M(M^\vee)^T=dI_d,
	\]
	and taking the \((a,b)\)-entry of the product gives
	\begin{equation}
		\sum_{j=1}^d w_j^{a-b}
		=
		d\delta_{ab},
		\quad a,b\in A.
		\label{eq:dual-power-sums}
	\end{equation}
We proceed by contradiction. To clarify the argument, we complete the proof in two steps. 

\

\noindent{\bf Step 1: Modulus stratification and layer matrix spectral analysis}

\

	Multiplying all \(w_j\) by the same positive constant does not change
	\eqref{eq:inverse-orthogonality}. We may therefore assume that their
	smallest modulus is \(1\). Write the distinct moduli as
	\[
	1=e^{t_1}<e^{t_2}<\cdots<e^{t_s},
	\qquad
	0=t_1<t_2<\cdots<t_s=T.
	\]
	If \(s=1\), the conclusion is immediate. Suppose that \(s\ge2\). Set
	\[
	I_r=\{j:|w_j|=e^{t_r}\},
	\qquad
	m_r= \# I_r,
	\qquad
	\sum_{r=1}^s m_r=d.
	\]
	For \(j\in I_r\), write
	\(
	w_j=e^{t_r}u_j\) with \(
 |u_j|=1.
	\)
	Define
	\[
	E_r=(u_j^a)_{a\in A,\,j\in I_r},
	\qquad
	G_r=E_rE_r^*.
	\]
	where $E_r$ is a $d \times m_r$ matrix and $G_r$ is a $d \times d$ matrix.  
	
	If \(i,j\in I_r\), then \(w_j/w_i=u_j/u_i\), and hence
	\[
	(E_r^*E_r)_{ij}
	=
	\sum_{a\in A}\overline{u_i^a}u_j^a
	=
	\sum_{a\in A}
	\left(\frac{u_j}{u_i}\right)^a
	=
	d\delta_{ij}.
	\]
	Therefore
	\begin{equation}
		E_r^*E_r=dI_{m_r}.
		\label{eq:Er-orthogonality}
	\end{equation}
	The matrix \(G_r\) is Hermitian positive semidefinite, since for every
	\(x\in\mathbb C^d\),
	\[
	x^*G_rx
	=
	x^*E_rE_r^*x
	=
	\|E_r^*x\|^2
	\ge0.
	\]
	Moreover,
	\[
	G_r^2
	=
	E_r(E_r^*E_r)E_r^*
	=
	dE_rE_r^*
	=
	dG_r.
	\]
	Thus every eigenvalue \(\lambda\) of \(G_r\) satisfies
	\[
	\lambda^2=d\lambda,
	\]
	so \(\lambda\in\{0,d\}\).
	Equation \eqref{eq:Er-orthogonality} shows that the \(m_r\) columns of
	\(E_r\) are linearly independent, and hence
	\(
	\operatorname{rank}E_r=m_r.
	\)
	We also have
	\begin{equation}
	\operatorname{rank}G_r
	=
	\operatorname{rank}E_r
	=
	m_r.
	\label{eq:Gr-rank}
	\end{equation}
	Consequently, \(G_r\) has exactly \(m_r\) nonzero eigenvalues, counted
	with multiplicity, and all of them are equal to \(d\). Furthermore,
	\begin{equation}
		(G_r)_{aa}
		=
		\sum_{j\in I_r}|u_j^a|^2
		=
		m_r,
		\quad a\in A.
		\label{eq:Gr-diagonal}
	\end{equation}
	For \(n\in\mathbb Z\), let
	\(
	U_r(n)=\sum_{j\in I_r}u_j^n.
	\)
	Then
	\begin{equation}
		(G_r)_{ab}=U_r(a-b),
		\quad a,b\in A.
		\label{eq:Gr-power-sums}
	\end{equation}
	Fix \(a,b\in A\) with \(a\ne b\), and set \(n=a-b\ne0\).
	By \eqref{eq:dual-power-sums},
	\begin{equation}
		\sum_{r=1}^s e^{t_rn}U_r(n)=0.
		\label{eq:positive-radial-relation}
	\end{equation}
Note that \(
U_r(-n)=\overline{U_r(n)}.
\)	Interchanging \(a\) and \(b\), then taking complex conjugates, gives
	\begin{equation}
		\sum_{r=1}^s e^{-t_rn}U_r(n)=0,
		\label{eq:negative-radial-relation}
	\end{equation}
	Subtracting \eqref{eq:negative-radial-relation} from
	\eqref{eq:positive-radial-relation}, and using \(t_1=0\) and \(t_s=T\),
	we obtain
	\begin{equation}
		U_s(n)
		=
		-\sum_{r=2}^{s-1}
		\frac{\sinh(t_rn)}{\sinh(Tn)}\,U_r(n).
		\label{eq:top-layer-relation}
	\end{equation}

\

\noindent {\bf Step 2: Unified contradiction for all \(s\ge 2\)}.

\

	If \(s=2\), the sum on the right-hand side of \eqref{eq:top-layer-relation} is empty. Hence
	\[
	U_2(a-b)=0
	\quad
	(a,b\in A,\ a\ne b).
	\]
	This together with \eqref{eq:Gr-power-sums} and
	\eqref{eq:Gr-diagonal} gives
\(
	G_2=m_2I_d,
	\)	
 	which contradicts with \(\rank(G_2)=m_2\) in \eqref{eq:Gr-rank}. Therefore \(s=2\) is impossible.
 	
 	\
 	
	Suppose now that \(s\ge3\). For \(2\le r\le s-1\), define
	\[
	C_r
	=
	\bigl(k_{t_r,T}(a-b)\bigr)_{a,b\in A},
	\]
	where
	\[
	k_{t,T}(x)
	=
	\begin{cases}
		\dfrac{\sinh(tx)}{\sinh(Tx)},&x\ne0,\\[6pt]
		\dfrac{t}{T},&x=0.
	\end{cases}
	\]
	By Lemma~\ref{lem:hyperbolic-sine-kernel}, each \(C_r\) is Hermitian positive
	semidefinite.
	Consider
	\[
	L
	=
	G_s+\sum_{r=2}^{s-1}C_r\circ G_r,
	\]
	where \(\circ\) denotes the Hadamard product. If \(a\ne b\), then by
	\eqref{eq:Gr-power-sums} and \eqref{eq:top-layer-relation},
	\[
	\begin{aligned}
		L_{ab}
		&=
		U_s(a-b)
		+
		\sum_{r=2}^{s-1}
		\frac{\sinh(t_r(a-b))}
		{\sinh(T(a-b))}
		\,U_r(a-b)=0.
	\end{aligned}
	\]
	Thus every off-diagonal entry of \(L\) is zero.
	For a diagonal entry, \eqref{eq:Gr-diagonal} and
	\(k_{t_r,T}(0)=t_r/T\) give
	\[
	(C_r\circ G_r)_{aa}
	=
	(C_r)_{aa}(G_r)_{aa}
	=
	\frac{t_r}{T}m_r.
	\]
	In particular, the diagonal entries of \(C_r\circ G_r\) are generally
	not zero. Hence
	\[
	L_{aa}
	=
	m_s+\sum_{r=2}^{s-1}\frac{t_r}{T}m_r
	=:\lambda,
	\quad a\in A.
	\]
	It follows that
	\begin{equation}
		G_s+\sum_{r=2}^{s-1}C_r\circ G_r
		=
		\lambda I_d.
		\label{eq:matrix-layer-identity}
	\end{equation}
	By the Schur product theorem,
	\[
	C_r\circ G_r\succeq0.
	\]
	Therefore \eqref{eq:matrix-layer-identity} implies
	\[
	\lambda I_d-G_s
	=
	\sum_{r=2}^{s-1}C_r\circ G_r
	\succeq0.
	\]
	Since \(m_s\ge1\), the matrix \(G_s\) has \(d\) as an eigenvalue.
	Taking a corresponding unit eigenvector \(v\), we obtain
	\[
	0
	\le
	v^*(\lambda I_d-G_s)v
	=
	\lambda-d,
	\]
	and hence
	\(
	\lambda\ge d.
	\)
	On the other hand, since \(0<t_r/T<1\),
	\[
	\lambda
	\le
	m_s+\sum_{r=2}^{s-1}m_r
	=
	d-m_1
	<
	d,
	\]
	which is a contradiction.
	Thus \(s\ge2\) is impossible, and all \(w_j\) have the same modulus.
\end{proof}

	\section{Proof of Theorem~\ref{thm:finite-rationality} and Theorem~\ref{thm:bounded-rationality}}
	\label{sec:bounded-rationality}

With the modulus rigidity result for orthogonal exponential systems in Theorem \ref{thm:concyclic-rigidity} established, we now return to prove our main structural theorem for spectra of finite integer sets, Theorem \ref{thm:finite-rationality}. The rigidity property from Theorem \ref{thm:concyclic-rigidity} will be combined with basic Galois theory and Kronecker’s classical theorem on algebraic integers to force every spectral element to be rational modulo one.
We can now complete the proof of the finite rationality theorem.

\vspace{0.2cm}

\begin{proof}[\bf Proof of Theorem~\ref{thm:finite-rationality}]
	The case \(d=1\) is trivial. Assume \(d\geq2\), and translate \(A\)
	by an integer so that \(\min A=0\). This translation does not alter the zero set
	of \(F_A\), since for \(A'=A-c\),
	\[
	F_{A'}(t)=e^{-2\pi ict}F_A(t).
	\]
	Define the generating polynomial 
	\[
	P_A(X)=\sum_{a\in A}X^a,
	\]
which	is a monic polynomial in \(\mathbb Z[X]\) with constant term \(1\) since \(0\in A\).	
Set \(z_j=e^{2\pi i \gamma_j}\in \mathbb{T}\) for \(0\le j<d\).	Note that \(z_0=1\). Taking \(i=0\) in
	\eqref{eq:finite-inverse-orthogonality}, we obtain
	\[
	P_A(z_j)=0,
	\quad j\neq0.
	\]
	Hence every \(z_j\) is an algebraic integer by definition.

\	

	Let \(K\) denote the normal closure (splitting field) of
	\(\mathbb Q(z_0,\ldots,z_{d-1})\) over \(\mathbb Q\). By definition of a normal extension, every irreducible polynomial over  \(\mathbb Q\) with one root in \(K\) factors completely into linear terms inside \(K\). In particular, all algebraic conjugates of any element of \(\mathbb Q(z_0,\ldots,z_{d-1})\) are contained in \(K\). Take any automorphism
	\(
	\sigma\in\operatorname{Gal}(K/\mathbb Q),
	\)
	the Galois group of the normal extension \(K/\mathbb Q\). Define \(w_j=\sigma(z_j)\). Field automorphisms preserve finite sums, integer powers and complex multiplication, so applying \(\sigma\) to
both sides of	\eqref{eq:finite-inverse-orthogonality} gives
	\[
	\sum_{a\in A}
	\left(\frac{w_j}{w_i}\right)^a
	=
	d\delta_{ij}, \quad 0\le i, j<d.
	\]
	By Theorem~\ref{thm:concyclic-rigidity},
	\(
	|w_0|=\cdots=|w_{d-1}|.
	\)
	Since
	\[
	w_0=\sigma(z_0)=\sigma(1)=1,
	\]
	we have
	\(
	|\sigma(z_j)|=1
	\)
	for every \(j\) and every
	\(\sigma\in\operatorname{Gal}(K/\mathbb Q)\).
	
	\
	
Since \(K/\mathbb Q\) is normal, Galois group \(\operatorname{Gal}(K/\mathbb Q)\) acts transitively on the roots of the minimal polynomial of any \(x\in \mathbb Q(z_0,\ldots,z_{d-1})\) (For full justification of this algebraic argument, see Appendix B). Thus every algebraic conjugate of
	\(z_j\) takes the form \(\sigma(z_j)\) for some \(\sigma\in\operatorname{Gal}(K/\mathbb Q)\). Recall \(
	|\sigma(z_j)|=1
	\) for all 	\(\sigma\in\operatorname{Gal}(K/\mathbb Q)\), so every conjugate of
	\(z_j\) lies on the unit circle .  Kronecker's theorem implies that
	\(z_j\) is a root of unity. Hence there exists \(N_j\in\mathbb N\)
	such that
	\(
	z_j^{N_j}=1,
	\)
	or equivalently,
	\[
	N_j\gamma_j\equiv 0 \pmod {1}.
	\]
	Therefore \(\gamma_j\in\mathbb Q/\mathbb Z\) for every \(j\).
	Taking
	\[
	N=\operatorname{lcm}(N_0,\ldots,N_{d-1}),
	\]
	we conclude that
	\(
	\Gamma\subset\frac1N\mathbb Z/\mathbb Z.
	\)
\end{proof}	

\vspace{0.2cm}

	We now combine the fiberization theorem with the finite-dimensional
	rationality result to prove Theorem~\ref{thm:bounded-rationality}.

\vspace{0.2cm}
	
	\begin{proof}[\bf Proof of Theorem~\ref{thm:bounded-rationality}]
		As established in Section~\ref{sec:fiberization}, there exist an
		integer \(T\geq1\) and a finite set \(S\subset[0,T)\) with
		\(0\in S\) such that
		\[
		\Lambda=S+T\mathbb Z.
		\]
		By Theorem~\ref{thm:fiberization}, for almost every
		\(x\in[0,1/T)\), the finite set
		\[
		A_x
		=
		\left\{
		n\in\mathbb Z:
		x+\frac{n}{T}\in\Omega
		\right\}
		\]
		has a spectrum
		\[
		\frac{S}{T}
		=
		\left\{
		\frac{s}{T}:s\in S
		\right\}
		\subset\mathbb R/\mathbb Z.
		\]
		Since \(0\in S/T\), Theorem~\ref{thm:finite-rationality} yields
		\[
		\frac{S}{T}\subset\mathbb Q/\mathbb Z.
		\]
		As \(S/T\subset[0,1)\), it follows that \(s/T\in\mathbb Q\) for every
		\(s\in S\), and hence \(S\subset\mathbb Q\). Since \(T\in\mathbb N\),
		we conclude that
		\[
		\Lambda=S+T\mathbb Z\subset\mathbb Q.
		\]
	\end{proof}
	
	\begin{remark}
		The condition \(0\in\Lambda\) is only a normalization. Indeed, if
		\(\lambda_0\in\Lambda\), then \(\Lambda-\lambda_0\) is again a
		spectrum of \(\Omega\), and therefore
		\[
		\Lambda-\lambda_0\subset\mathbb Q.
		\]
		More generally, 
		for a bounded spectral set \(\Omega\) of arbitrary positive measure, this yields
		\[
		|\Omega|(\Lambda-\lambda_0)\subset\mathbb Q
		\quad(\lambda_0\in\Lambda).
		\]
	\end{remark}

\

\appendix
\section{A Fourier integral fomula}

\begin{lemma}
	For \( 0<a<\pi\) and \(s\in\mathbb{R}\), we have
	\begin{equation}
		\int_{\mathbb R}
		\frac{e^{isu}}{\cosh u+\cos a}\,du
		=
		\frac{2\pi\sinh(as)}
		{\sin a\,\sinh(\pi s)},
		\label{eq:hyperbolic-fourier}
	\end{equation}
	where the value at \(s=0\) is understood by continuity.
\end{lemma}

\begin{proof}

	Assume first that \(s>0\), and consider 
	\[
	\ f(z)=\frac{e^{isz}}{\cosh z+\cos a}
	\]
	integrated over the positively oriented rectangular contour with vertices 
	\[
	-R,\quad R,\quad R+2\pi i,\quad -R+2\pi i.
	\]
	The two poles of \(f\) in the strip \(0<\operatorname{Im}z<2\pi\) are
	\[
	z_1=i(\pi-a),
	\qquad
	z_2=i(\pi+a),
	\]
	and they are simple. Their residues are
	\[
	\operatorname{Res}(f;z_1)
	=
	\frac{e^{-s(\pi-a)}}{i\sin a},
	\qquad
	\operatorname{Res}(f;z_2)
	=
	-\frac{e^{-s(\pi+a)}}{i\sin a}.
	\]	
	Let
	\[
	I_R(s)=\int_{-R}^{R}
	\frac{e^{isx}}{\cosh x+\cos a}\,dx.
	\]
	denote the integral over the real segment of the contour.
	On the
	upper horizontal edge \(z=x+2\pi i\),  using the identities 
	\[
	\cosh(x+2\pi i)=\cosh x,
	\qquad
	e^{is(x+2\pi i)}=e^{-2\pi s}e^{isx}.
	\]
	and noting that the orientation is
	from \(R+2\pi i\) to \(-R+2\pi i\), we obtain
	\[
	\int_{\mathrm{upper}}f(z)\,dz
	=
	-e^{-2\pi s}I_R(s),
	\]
	For the vertical sides \(z=\pm R+iy\) with \(0\leq y\leq2\pi\).
	Since
	\[
	|e^{isz}|=e^{-sy}\leq1
	\]
	and
	\[
	|\cosh(R+iy)|^2
	=
	\sinh^2R+\cos^2y
	\geq\sinh^2R,
	\]
	we have
	\[
	|\cosh(\pm R+iy)+\cos a|
	\geq \sinh R-1.
	\]
	Hence each vertical integral is bounded in modulus by
	\(
	\frac{2\pi}{\sinh R-1},
	\)
	and therefore tends to zero as \(R\to\infty\).	
	The residue theorem now gives
	\[
	(1-e^{-2\pi s})
	\int_{\mathbb R}
	\frac{e^{isu}}{\cosh u+\cos a}\,du
	=
	2\pi i
	\left(
	\frac{e^{-s(\pi-a)}}{i\sin a}
	-
	\frac{e^{-s(\pi+a)}}{i\sin a}
	\right),
	\]
	that is,
	\[
	(1-e^{-2\pi s})
	\int_{\mathbb R}
	\frac{e^{isu}}{\cosh u+\cos a}\,du
	=
	\frac{4\pi e^{-\pi s}\sinh(as)}{\sin a}.
	\]
	Formula \eqref{eq:hyperbolic-fourier} follows for \(s>0\) since 
	\(
	1-e^{-2\pi s}
	=
	2e^{-\pi s}\sinh(\pi s)
	\). The case
	\(s<0\) follows by replacing \(u\) with \(-u\), since both sides are
	even functions of \(s\).
	For \(s=0\), note that
	\[
	\left|
	\frac{e^{isu}}{\cosh u+\cos a}
	\right|
	\leq
	\frac{1}{\cosh u+\cos a},
	\]
	and the function on the right is integrable on \(\mathbb R\), since
	it decays exponentially as \(|u|\to\infty\). Thus, by the dominated
	convergence theorem,
	\[
	\lim_{s\to0}
	\int_{\mathbb R}
	\frac{e^{isu}}{\cosh u+\cos a}\,du
	=
	\int_{\mathbb R}
	\frac{du}{\cosh u+\cos a}.
	\]
	On the other hand,
	\[
	\lim_{s\to0}
	\frac{2\pi\sinh(as)}
	{\sin a\,\sinh(\pi s)}
	=
	\frac{2a}{\sin a}.
	\]
	Hence \eqref{eq:hyperbolic-fourier} also holds at \(s=0\) by
	continuous extension.
	
\end{proof}

	\section{Galois automorphisms for orthogonality identities}
	
All standard foundational results of Galois theory referenced below can be found in \cite{Lang2002}. This appendix provides targeted, self-contained justification for the algebraic manipulations used in the proof of Theorem \ref{thm:finite-rationality}, intended for readers whose primary research background lies in harmonic analysis. We omit redundant textbook-style general definitions and focus only on statements directly required for our spectral rigidity argument.

	\subsection{Invariance properties of Galois automorphisms}
	Let $K/\Q$ be a field extension and  
	$\sigma\in\Gal(K/\Q)$. By definition,   $\sigma$ is a $\Q$-fixing field automorphism, hence it preserves addition, multiplication, and rational scalars. Standard field automorphism properties imply 
	\begin{equation*}\label{eq:sigma-quotient}
		\sigma\left(\frac{x}{y}\right)
		=\frac{\sigma(x)}{\sigma(y)} \quad (y\ne 0), 
		\ \ \   \ \ \sigma(x^m)=\sigma(x)^m
		\quad(m\in\Z,\ x\ne0\text{ if }m<0).
	\end{equation*}
Consequently, applying \(\sigma\) to any finite sum identity with rational coefficients preserves equality. In particular, the orthogonality relation
\begin{equation*}
	\sum_{a\in A}
	\left(\frac{z_j}{z_i}\right)^a
	=
	d\delta_{ij},
	\qquad 0\leq i,j< d.
\end{equation*}
remains valid after substituting \(w_k=\sigma(z_k)\).

\subsection{Normal closure and realization of algebraic conjugates}

The Galois invariance argument for orthogonality relations and the realization of algebraic conjugates via Galois automorphisms below follow the framework introduced by Konyagin and \L aba \cite{KonyaginLaba2003} for integer spectral sets, whose arguments were restricted to irreducible polynomials. In contrast, our analysis applies to general polynomials without irreducibility constraints. 

\vspace{0.2cm}

Let \(A\subset \mathbb{Z}\) and \(\Gamma=\{\gamma_0,\dots,\gamma_{d-1}\}\) be as in Theorem~\ref{thm:finite-rationality}, and set \(z_j := e^{2\pi i \gamma_j}\in \mathbb{T}\). Denote \(m_{z_j}(x) \in \mathbb{Q}[x]\) be the minimal polynomial of \(z_j\) over \(\mathbb{Q}\). Define
\[
K := \mathrm{Spl}\!\left(\prod_{k=0}^{d-1} m_{z_k}(x)\right),
\]
the splitting field  over \(\mathbb{Q}\) of the product of minimal polynomials for all \(z_k\).	The field $K/\Q$  is finite and normal, and equals the normal closure of \( \mathbb{Q}(z_0,\dots,z_{d-1})\).
	
	\begin{lemma}[Transitive action of Galois group on roots]
		\label{lem:main}
		 Fix \(0\le j<d\), and let \(m_{z_j}(x) \in \mathbb{Q}[x]\) denote the minimal polynomial of \(z_j\) over \(\mathbb{Q}\). The Galois group  \(\mathrm{Gal}(K/\mathbb{Q})\) acts transitively on the root set of each irreducible factor \(m_{z_k}(x)\) of \(\prod_{k=0}^{d-1} m_{z_k}(x)\). In particular, for any algebraic conjugate 
		 \(\beta\) of \(z_j\) over \(\mathbb{Q}\),  there exists an automorphism
		\(
		\sigma \in \mathrm{Gal}(K/\mathbb{Q})
		\)
		such that \(\sigma(z_j) = \beta\).
		\end{lemma}	
		
		\begin{proof}
			Since \(K\) is the splitting field of \(\prod_{k=0}^{d-1} m_{z_k}(x)\) over \(\mathbb{Q}\), each minimal polynomial \(m_{z_k}\) splits completely in \(K[x]\). In particular, the conjugate \(\beta\) of \(z_j\) is a root of \(m_{z_j}\), hence \(\beta \in K\).
			Define a map 
			\[
			\tau: \mathbb{Q}(z_j) \longrightarrow \overline\Q, \qquad \tau(f(z_j)) := f(\beta), \quad \forall f(x) \in \mathbb{Q}[x].
			\]
		 This map is well-defined. Suppose \(f(z_j)=g(z_j)\), then \((f-g)(z_j)=0\). Since \(m_{z_j}\) is the minimal polynomial of \(z_j\), we have \(m_{z_j}\) divides \(f-g\). As \(\beta\) is also a root of \(m_{z_j}\), it follows that \((f-g)(\beta)=0\), hence \(f(\beta)=g(\beta)\). The map clearly preserves addition and multiplication, and sends \(1\) to \(1\). It is injective. If \(\tau(f(z_j))=0\), then \(f(\beta)=0\). Since \(z_j\) and \(\beta\) share the same minimal polynomial \(m_{z_j}\), we have \(m_{z_j} \mid f\), so \(f(z_j)=0\) in \(\mathbb{Q}(z_j)\). Thus \(\tau\) is a \(\mathbb{Q}\)-embedding.
		 
		 \
			
			The extension \(K/\mathbb{Q}(z_j)\) is finite since \(K/\mathbb{Q}\) is finite as a splitting field of a polynomial over \(\mathbb{Q}\). By the standard Isomorphism Extension Theorem, we can extend \(\tau\) to an embedding
			\[
			\widetilde{\tau}: K \longrightarrow \overline{\mathbb{Q}}.
			\]
			This extension fixes \(\mathbb{Q}\), restricts to \(\tau\) on \(\mathbb{Q}(z_j)\), and in particular satisfies \(\widetilde{\tau}(z_j)=\tau(z_j)=\beta\).
			
			\
			
			For every $x\in K$, the minimal polynomial of $x$ over $\Q$ splits completely in $K$ due to nomality of \(K/\mathbb{Q}\).  The image $\widetilde{\tau}(x)$ is another root of this polynomial, so $\widetilde{\tau}(x)\in K$.  Hence $\widetilde{\tau}(K)\subseteq K$. We have therefore obtained an injective field homomorphism
			\[
			\widetilde{\tau}: K \longrightarrow K
			\]
			fixing \(\mathbb{Q}\) pointwise. Since \(K/\mathbb{Q}\) is finite, \(K\) is a finite-dimensional vector space over \(\mathbb{Q}\). The map \(\widetilde{\tau}\) is an injective \(\mathbb{Q}\)-linear transformation from this space to itself. By the rank-nullity theorem, an injective linear map on a finite-dimensional vector space is necessarily surjective. Hence \(\widetilde{\tau}\) is bijective. Being a bijective field homomorphism fixing \(\mathbb{Q}\), it is a field automorphism of \(K\) over \(\mathbb{Q}\) by definition. Thus
			\[
			\widetilde{\tau} \in \mathrm{Gal}(K/\mathbb{Q}).
			\]
		Setting \(\sigma=\widetilde{\tau}\), we obtain  \(\sigma(z_j)=\beta\). This confirms the transitive action of \(\mathrm{Gal}(K/\mathbb{Q})\) on the roots of the irreducible polynomial \(m_{z_j}\).
		
		\end{proof}

\subsection{Consequence for unit-circle containment of algebraic conjugates}	
Recall that $z_0=1$, so \(\sigma(z_0)=\sigma(1)=1\) for every 	\(\sigma\in\operatorname{Gal}(K/\mathbb Q)\). Theorem~\ref{thm:concyclic-rigidity}  in the main text guarantees 
	\(
	|w_0|=\cdots=|w_{d-1}|
	\)
	for \(w_k=\sigma(z_k)\), which implies that
	\(
	|\sigma(z_j)|=1
	\)
	for every \(j\) and every
	\(\sigma\in\operatorname{Gal}(K/\mathbb Q)\).

	By Lemma \ref{lem:main}, every algebraic conjugate of
	\(z_j\) is of the form \(\sigma(z_j)\) for some \(\sigma\in\operatorname{Gal}(K/\mathbb Q)\). Therefore, every conjugate of
	\(z_j\) lies on the unit circle \(\mathbb{T}\). Kronecker's classical theorem on algebraic integers with all conjugates on the unit circle forces each
	\(z_j\) to be a root of unity, completing the algebraic step of Theorem \ref{thm:finite-rationality}.

\vspace{0.2cm}
	
\noindent\textbf{Acknowledgements}
The authors are grateful to Professor Chun-Kit Lai from San Francisco State University for numerous insightful conversations and detailed revision suggestions throughout the preparation of this manuscript, as well as for bringing the references \cite{BoseMadan2018,  DutkayJorgensen, KonyaginLaba2003,Lang2002} to our attention. The research is supported by the NSFC grant 12271194.

	\begingroup
	\small
\begin{thebibliography}{99}
	
		\bibitem{BoseMadan2011}
	D.~Bose and S.~Madan,
	Spectrum is periodic for \(n\)-intervals,
	\emph{J. Funct. Anal.} \textbf{260} (2011), no.~1, 308--325.

	\bibitem{BoseMadan2018}
D.~Bose and S.~Madan,
 On the Rationality of the Spectrum, \emph{J Fourier Anal Appl} \textbf{24}  (2018), 1037--1047.

	\bibitem{DutkayHaussermannLai2019}
	D.~E. Dutkay, J.~Haussermann, and C.-K.~Lai,
	Hadamard triples generate self-affine spectral measures,
	\emph{Trans. Amer. Math. Soc.} \textbf{371} (2019), no.~2,
	1439--1481.

	\bibitem{DutkayJorgensen}
	D. E. Dutkay, and P. E. T. Jorgensen,
	On the universal tiling conjecture in dimension one,
	\emph{J. Fourier Anal. Appl.} \textbf{19} (2013), 467--477.
	
	\bibitem{DutkayLai2014}
	D.~E. Dutkay and C.-K.~Lai,
	Some reductions of the spectral set conjecture to integers,
	\emph{Math. Proc. Cambridge Philos. Soc.} \textbf{156} (2014), no.~1, 123--135.
	
		\bibitem{FarkasMatolcsiMora2006}
	B.~Farkas, M.~Matolcsi, and P.~M\'ora,
	On Fuglede's conjecture and the existence of universal spectra,
	\emph{J. Fourier Anal. Appl.} \textbf{12} (2006), no.~5, 483--494.
	
	\bibitem{Fuglede1974}
	B.~Fuglede,
	Commuting self-adjoint partial differential operators and a group theoretic problem,
	\emph{J. Funct. Anal.} \textbf{16} (1974), no.~1, 101--121.
	
		\bibitem{Horn:2013:MatrixAnalysis}
	R. A. Horn and C. R. Johnson, 
	Matrix Analysis,
	Cambridge University Press, Cambridge, 2nd edition, 2013.
	
	 \bibitem{IosevichKolountzakis2013}
	A.~Iosevich and M.~N. Kolountzakis,
	Periodicity of the spectrum in dimension one,
	\emph{Anal. PDE} \textbf{6} (2013), no.~4, 819--827.
	
		\bibitem{KissMalikiosisSomlaiVizer2022}
	G.~Kiss, R.~D. Malikiosis, G.~Somlai, and M.~Vizer,
	Fuglede's conjecture holds for cyclic groups of order \(pqrs\),
	\emph{J. Fourier Anal. Appl.} \textbf{28} (2022),
	Article No.~79, 23 pp.
	
	\bibitem{Kolountzakis2012}
	M.~N. Kolountzakis,
	Periodicity of the spectrum of a finite union of intervals,
	\emph{J. Fourier Anal. Appl.} \textbf{18} (2012), no.~4, 673--676.
	
		\bibitem{Kolountzakis2024}
	M.~N. Kolountzakis,
	Orthogonal Fourier analysis on domains,
	\emph{Expo. Math.} (2024), Article 125629.
	
	\bibitem{KolountzakisLevMatolcsi2023}
	M.~N. Kolountzakis, N.~Lev, and M.~Matolcsi,
	Spectral sets and weak tiling,
	\emph{Sampl. Theory Signal Process. Data Anal.}
	\textbf{21} (2023), Paper No.~31, 21 pp.

    	\bibitem{KolountzakisMatolcsi2006}
    M.~N. Kolountzakis and M.~Matolcsi,
    Complex Hadamard matrices and the spectral set conjecture,
    \emph{Collect. Math.} Vol.~Extra (2006), 281--291.
    
    	\bibitem{KonyaginLaba2003}
  S.~Konyagin and I.~\L aba,  Spectra of certain types of polynomials and tiling of integers with translates of finite sets,
    \emph{J. Number Theory} \textbf{103} (2003), 267--280.
    
    	\bibitem{Laba2001}
    I.~\L aba,
    Fuglede's conjecture for a union of two intervals,
    \emph{Proc. Amer. Math. Soc.} \textbf{129} (2001), no.~10,
    2965--2972.
    
    \bibitem{Laba2002}
    I.~\L aba,
    The spectral set conjecture and multiplicative properties of roots
    of polynomials,
    \emph{J. London Math. Soc.} (2) \textbf{65} (2002), no.~3,
    661--671.
    
     \bibitem{LabaLondner2023}
    I.~\L aba and I. Londner, The {C}oven-{M}eyerowitz tiling conditions for 3 odd prime
    factors,
    \emph{Invent. Math.}  \textbf{232} (2023), no.~1,
    365–470.
    
    	\bibitem{LagariasWang1997}
    J.~C. Lagarias and Y.~Wang,
    Spectral sets and factorizations of finite abelian groups,
    \emph{J. Funct. Anal.} \textbf{145} (1997), no.~1, 73--98.
    
    \bibitem{Landau1967}
    H.~J. Landau,
    Necessary density conditions for sampling and interpolation of certain entire functions,
    \emph{Acta Math.} \textbf{117} (1967), 37--52.
    
 \bibitem{Lang2002}
 S.~Lang, Algebra, Graduate Texts in Mathematics, vol. 211, Third edition, Springer-Verlag, New York, 2002.
   
	\bibitem{LevMatolcsi2022}
	N.~Lev, M.~Matolcsi, 
	The Fuglede conjecture for convex domains is true in all dimensions.
	\emph{Acta Math.} \textbf{228} (2022), No.~2, 385–420.
	
		\bibitem{Malikiosis2022}
	R.~D. Malikiosis,
	On the structure of spectral and tiling subsets of cyclic groups,
	\emph{Forum Math. Sigma} \textbf{10} (2022),
	Paper No.~23, 42 pp.
	
	\bibitem{MalikiosisKolountzakis2017}
	R.~D. Malikiosis and M.~N. Kolountzakis,
	Fuglede's conjecture on cyclic groups of order \(p^nq\),
	\emph{Discrete Anal.} (2017), Paper No.~12, 16 pp.
	
	\bibitem{Matolcsi2005}
	M.~Matolcsi,
	Fuglede's conjecture fails in dimension \(4\),
	\emph{Proc. Amer. Math. Soc.} \textbf{133} (2005), no.~10, 3021--3026.

	\bibitem{Shi2019}
	R.~Shi,
	Fuglede's conjecture holds on cyclic groups
	\(\mathbb Z_{pqr}\),
	\emph{Discrete Anal.} (2019), Paper No.~14, 14 pp.

     	\bibitem{TadejZyczkowski2006}
     W.~Tadej and K.~\.{Z}yczkowski,
     A concise guide to complex Hadamard matrices,
     \emph{Open Syst. Inf. Dyn.} \textbf{13} (2006), no.~2, 133--177.

	\bibitem{Tao2004}
T.~Tao,
Fuglede's conjecture is false in \(5\) and higher dimensions,
\emph{Math. Res. Lett.} \textbf{11} (2004), no.~2--3, 251--258.

	\bibitem{Zhou2024}
W.~Zhou,
On rationality of spectrums for spectral sets in $\mathbb{R}$,
\emph{J. Funct. Anal.} \textbf{286} (2024), no.~10, 13pp.

	\endgroup

	\medskip
	\begingroup
	\small
	\noindent\textsc{Xiao-Ye Fu}\\
	Hubei Key Laboratory of Mathematical Sciences, College of Mathematics and Statistics,
	Central China Normal University, Wuhan, Hubei 430079, China\\
	\textit{Email address}: \href{mailto:xiaoyefu@ccnu.edu.cn}{xiaoyefu@ccnu.edu.cn}
	
	\smallskip
	\noindent\textsc{Zi-Jian Song}\\
	Hubei Key Laboratory of Mathematical Sciences, College of Mathematics and Statistics,
	Central China Normal University, Wuhan, Hubei 430079, China\\
	\textit{Email address}: \href{mailto:songzijian2025@outlook.com}{songzijian2025@outlook.com}

	\endgroup
	
\end{document}